\documentclass[submission,copyright,creativecommons]{eptcs}
\providecommand{\event}{QPL 2016}
\newcommand{\sss}[1]{[\textbf{SS: }#1]}

\usepackage{amsmath,amsthm}
\usepackage{amssymb}
\usepackage{mathtools}
\usepackage[all,cmtip,2cell]{xy}
\UseTwocells
\usepackage{appendix}
\usepackage{comment}
\usepackage{multicol}
\newtheorem{theorem}{Theorem}
\newtheorem{lemma}{Lemma}
\newtheorem{corollary}{Corollary}

\newtheorem{proposition}{Proposition}
\newtheorem*{theorem*}{Theorem}

\newcommand{\bra}[1]{\left\langle #1 \right|}
\newcommand{\ket}[1]{\left| #1 \right\rangle}

\newcommand{\C}{\mathbb{C}}
\newcommand{\R}{\mathbb{R}}
\newcommand{\N}{\mathbb{N}}

\newcommand{\cat}[1]{\mathbf{#1}}

\newcommand{\Cone}{\cat{Cone}}

\newcommand{\WStar}{\mathbf{W^*\text{-}Alg}}
\newcommand{\CStar}{\mathbf{C^*\text{-}Alg}}

\newcommand{\PNU}{\CStar_\mathrm{P}}
\newcommand{\CPNU}{\CStar_\mathrm{CP}}

\newcommand{\wPNU}{\WStar_\mathrm{P}}
\newcommand{\wCPNU}{\WStar_\mathrm{CP}}

\newcommand{\opp}[1]{#1^\mathbf{op}}
\newcommand{\sa}[1]{#1_{\mathrm{sa}}}

\newcommand{\new}{\operatorname{new}}

\newcommand{\mat}[1]{M_{#1}}

\newcommand{\Mat}{\mathbb{N}_\mathrm{Mat}}

\newcommand{\NatCP}{\mathbb{N}_\mathrm{CP}}

\newcommand{\CC}{\mathbb{C}}

\newcommand{\hide}[1]{}

\title{Infinite-Dimensionality in Quantum Foundations:\\ W*-algebras as Presheaves over Matrix Algebras}
\author{
Mathys Rennela
\institute{Radboud University\\
    Nijmegen, The Netherlands}
\and
Sam Staton
\institute{Oxford University\\
    Oxford, United Kingdom}
\and
Robert Furber
\institute{Aalborg University\\
    Aalborg, Denmark}
}
\def\titlerunning{Infinite-dimensionality in Quantum Foundations}
\def\authorrunning{M. Rennela, S. Staton \& R. Furber}

\newcommand{\Rpnz}{\mathbb{R}_{>0}}
\newcommand{\FDCP}{\mathbf{Fd}\CStar_{\mathrm{CP}}}
\newcommand{\CFDCP}{\mathbf{Fd}\mathbf{C}\CStar_{\mathrm{CP}}}
\newcommand{\Set}{\mathbf{Set}}

\begin{document}
\maketitle

\begin{abstract}
In this paper, W*-algebras are presented as canonical colimits of diagrams of matrix algebras and completely positive maps. In other words, matrix algebras are dense in W*-algebras.
\end{abstract}

\section*{Introduction}
\quad 
In the foundations of quantum mechanics and quantum computing, there is often a split between research using 
infinite dimensional structures and research using finite dimensional structures. 
On the one hand, in axiomatic quantum foundations there is often a focus on finite dimensional spaces and matrix mechanics
(e.g.~\cite{CQM,tull-OTP,hardy5,vicary-cat-fd-C*,chiribella-OTP,baez-rosetta-stone,coecke-axiomatic,coecke-martin-domains,coecke-meas-sums,quantum-weakest-preconditions,lee-barrett-GPT}),
and the same is true for circuit based quantum computing~(e.g.~\cite{feynman,nielsen-chuang}). 
On the other hand, infinite dimensional spaces arise naturally in subjects such as quantum field theory \cite{zeidler-QFT-QED},
and moreover the register space in a scalable quantum computer arguably has an infinite dimensional aspect (see e.g.~\cite{QC-w-qu-phot}),
which has led some researchers to use infinite dimensional spaces in the semantics of quantum programming languages~\cite{cho-qpl14-extended,rennela-mfps30,rennela-staton-mfps31,gielerak-sawerwain-generalised}.
The `spaces' in quantum theory are really non-commutative, so we understand them as W*-algebras, by analogy to Gelfand duality \cite[1.4]{connes}.\\

A natural question, then, is whether foundational research frameworks that focus on finite dimensional structures can approximate their infinite-dimensional counterparts.
In brief, the answer to this question is positive when one deals with W*-algebras. 
In detail, when we  focus on completely positive maps, as is usual in quantum foundations and quantum computation,
one can show that every infinite dimensional W*-algebra is a canonical colimit of matrix algebras. This characteristic is expressed in the following theorem, which constitutes our main result.
Note that it is about colimits in the opposite category of W*-algebras and completely positive maps, and therefore about limits in the category of W*-algebras and completely positive maps.

\begin{theorem*}
Let $\wCPNU$ be the category of W*-algebras together with completely positive maps. 
Let $\NatCP$ be the category whose objects are natural numbers, with $n$ considered as the algebra of $n\times n$ complex matrices, and completely positive maps between them. Let $\cat{Set}$ be the category of sets and functions.

The hom-set functor $\wCPNU(-,=): \opp{\wCPNU}\to [\NatCP, \cat{Set}]$ is full and faithful.
\end{theorem*}
Recall that the category of presheaves $[\NatCP, \cat{Set}]$ is a free colimit completion of $\opp{\NatCP}$ (e.g.~\cite[III.7]{maclane}). 
Recall too that a full-and-faithful functor is the same thing as a full subcategory, up to categorical equivalence. 
Thus we can say that every W*-algebra is a canonical colimit of matrix algebras.
A category theorist would say that the matrix algebras are
\emph{dense} in the W*-algebras (following \cite[Ch.~5]{kelly}), making an
analogy with topology (e.g. the rational numbers are dense
among the reals).
 We phrase the result in terms of the dual category $\opp{\wCPNU}$ instead of $\wCPNU$ 
with the idea that $\opp{\wCPNU}$ is a category of non-commutative spaces.

\newpage
\paragraph{Related ideas.} Our theorem is novel (as far as we can tell)
but the theme is related to various research directions.
\begin{itemize}
\item Density theorems occur throughout category theory. Perhaps the most famous situation is simplicial sets, which are 
functors $[\opp \Delta,\Set]$, where $\Delta$ is a category whose objects are natural numbers, with $n$ considered as the $n$-simplex. There is a restricted hom-functor $\cat{Top}\to[\opp \Delta,\Set]$; it is full and faithful up-to homotopy when one restricts $\cat{Top}$ to CW-complexes.
\item There is a long tradition of studying limits and colimits of *-homomorphisms, rather than completely positive maps. Notably, AF C*-algebras are limits of directed diagrams of finite-dimensional C*-algebras and *-homo\-morphisms~\cite{bratteli-AF-C*}.
\item In a dual direction, C*-algebras and *-homomorphisms form a locally presentable category~\cite{pelletier-rosicky} and so there exists a small dense set of C*-algebras with respect to *-homomorphisms. This dense set of C*-algebras has not been characterized explicitly, to our knowledge, but it is likely to already contain infinite dimensional C*-algebras.

\item Operator spaces and operator systems are generalizations of C*-algebras that still permit matrix constructions.
These are also related to the presheaf construction, as explained in Section~\ref{sec:cp-tvect}. 
\item In programming language theory, aside from quantum computation, the idea of defining computational constructs on dense subcategories is 
increasingly common (e.g.~\cite{mellies-segal}). 
\item Density also appears in quantum contextuality. 
  For example, the boolean algebras are dense in the
  effect algebras~\cite{staton-uijlen}, and compact Hausdorff spaces
  are dense in piecewise C*-algebras~\cite[Thm.~4.5]{flori-fritz}.
  This is a compelling way to study
  contextuality: the base category offers a classical perspective on
  the quantum situation. However, it is unclear how to study tensor products of C*/W*-algebras in this way. 
\item The Karoubi envelope of a category is a very simple colimit
  completion. The Karoubi envelope of $\NatCP$ contains the category of finite-dimensional C*-algebras and completely positive maps, as discussed in \cite{selinger-karoubi,heunen-kissinger-selinger}.
\item Pagani, Selinger and Valiron~\cite{pagani-selinger-valiron-popl14} used a free biproduct completion of $\NatCP$ to model 
higher-order quantum computation. It remains to be seen whether every object of their category can be thought of as a W*-algebra,
and whether their type constructions correspond to known constructions of W*-algebras.
\item Malherbe, Scott and Selinger~\cite{malherbe-scott-selinger} proposed to study quantum computation using presheaf categories $[\opp{\mathbf Q},\Set]$, where $\mathbf Q$ is a category related to $\NatCP$. Thus our result links their proposal for higher-order quantum computation with work based on operator algebra.
\end{itemize}

\section{Preliminaries on operator algebras}
\label{sec:preliminaries}

In this section we briefly recall some key concepts from operator algebra. See \cite{sakai,takesaki1} for a complete introduction.
We recall C*-algebras, which are, informally, non-commutative topological spaces, as a step towards W*-algebras, which are, informally, a non-commutative measure spaces \cite[1.4]{connes}. The positive elements of the algebras are thought of as observables, and so we focus on linear
maps that preserve positive elements. Completely positive maps are, roughly, positive maps that remain positive when quantum systems are combined. 

\paragraph{C*-algebras}
Recall that a (unital) C*-algebra
is a vector space over the field of complex numbers 
that also has multiplication, a unit and an involution,
satisfying associativity laws for multiplication, 
involution laws (e.g. $x^{**}=x$, $(xy)^*=y^*x^*$, $(\alpha x)^*=\bar \alpha (x^*)$) and
such that the norm of an element $x$ is given by the square root of the spectral radius of $x^*x$, which gives a Banach space.

\paragraph{Finite dimensional examples: qubits and bits}
A key source of examples of finite dimensional C*-algebras are
the algebras $\mat k$ of $k\times k$ complex matrices,
with matrix addition and multiplication,
and where involution is conjugate transpose.
In particular the set $\mat 1=\CC$ of complex numbers 
has a C*-algebra structure, and 
the $2\times 2$ matrices, $\mat 2$, contain
the observables of qubits.

Another example is the algebras of pairs of complex numbers, 
$\CC^2$, with componentwise addition and multiplication. 
This contains the observables of classical bits. 

\paragraph{Positive elements and positive maps}
An element $x \in A$
is \emph{positive} if it can be written in the form $x=y^* y$ for $y \in
A$. 
We denote by $A^+$ the set of positive elements of a C*-algebra $A$ and define the following.
\hide{An element $x \in A$ is self-adjoint if $x=x^*$. We write $\sa{A} \hookrightarrow A$ (resp. $A^+ \hookrightarrow
A$) for the subset of self-adjoint (resp. positive) elements of $A$,
which is a convex cone and thus induces a partial order structure: $x
\leq y$ if and only if $y-x \in A^+$.}

Let $f : A \rightarrow B$ be a linear map between the underlying vector spaces.
The map $f$ is \textit{positive} if it preserves positive elements and therefore restricts to a function $A^+ \to B^+$. A positive map $A \to \C$ will be called a state on $A$.

\paragraph{W*-algebras and normal maps}
In what follows, we will focus on W*-algebras, which are 
C*-algebras $A$ that have a predual, that is, such that there is a Banach space $A_*$ whose dual is isomorphic to $A$~\cite{sakai}.
The positive elements of a C*-algebra always form a partial order, with
$x\leq y$ if and only if $(y-x) \in A^+$. Moreover, in a W*-algebra, if a directed subset of $A$ has an upper bound, then it has a least upper bound.
It is natural to require that (completely) positive maps are moreover \emph{normal}, which means that they preserve such least upper bounds.

W*-algebras encompass all finite dimensional C*-algebras,
and also the algebras of bounded operators on any Hilbert space,
the function space $L^\infty(X)$ for any standard measure space $X$,
and the space $\ell^\infty(\N)$ of bounded sequences.
\paragraph{Matrix algebras and completely positive maps}
If $A$ is a C*-algebra then the $k\times k$ matrices valued in $A$
also form a C*-algebra, $\mat k(A)$, which is a W*-algebra if $A$ is.  For instance $\mat k(\CC)=\mat k$,
and $\mat k(\mat l)\cong \mat {k\times l}$.  Informally, we can think of the
W*-algebra $\mat k(A)$ as representing $k$ possibly-entangled copies of $A$.  This
can be thought of as a kind of tensor product: as a vector space
$\mat k(A)$ is a tensor product $\mat k(\CC)\otimes A$. 

Let $f : A \rightarrow B$ be a linear map between the underlying vector spaces.
The map $f$ is \textit{completely positive} if it is $n$-positive for every $n \in \N$, i.e. the map $\mat n(f):\mat n(A) \to \mat n(B)$ defined for every matrix $[x_{i,j}]_{i,j\leq n} \in \mat n(A)$ by $\mat n(f)([x_{i,j}]_{i,j\leq n})=[f(x_{i,j})]_{i,j\leq n}$ is positive for every $n \in \N$.

Completely positive maps and positive maps are related as follows: a positive map $f: A \to B$ of C*-algebras, for which $A$ or $B$ is commutative, is completely positive. Hence, every (sub-)state on a C*-algebra $A$ is completely positive.

We write $\wPNU$ for the category of W*-algebras and normal positive maps, 
and $\wCPNU$ for the category of W*-algebras and normal completely positive maps.

\section{Naturality and representations of complete positivity}
\label{sec:nat-cp}
In this section, we recall the categorical characterization of
completely positive maps as natural families of positive maps. This gives a technique for building representations of completely positive maps (see \cite{rennela-staton-mfps31} for more examples).\\

In the previous section we explained that for every W*-algebra $A$ the matrices valued in $A$ form a W*-algebra again.
This construction $(A,m)\mapsto M_m(A)$ 
is functorial. To make this precise, we introduce the category $\Mat$ of complex matrices:
the objects are non-zero natural numbers seen as dimensions, 
and the morphisms $m\to n$ are $m\times n$ complex matrices. Composition is matrix multiplication. 
(We remark that the category $\Mat$ is equivalent to the category $\cat{FdVect}_\C$ of
finite-dimensional complex vector spaces and linear maps, since every finite-dimensional vector space is isomorphic to $\C^n$. It is also equivalent to 
the category $\cat{FdHilb}$ of finite-dimensional Hilbert spaces and linear maps, since
every such space has a canonical inner product.)\\

The construction of matrices of elements of a W*-algebra can be made into 
a functor $\wCPNU\times \Mat \to \wPNU$.
It takes a pair $(A,m)$ to $\mat m(A)$ and a pair of morphisms
$(f,F):(A,m)\to (B,n)$ 
to the positive map $F^*{(f\!\_\,)}F:\mat m(A)\to\mat n(B)$.\\

We will consider this functor in curried form, 
$M:\wCPNU\to [\Mat,\wPNU]$.
It takes a W*-algebra $A$ to a functor,
i.e. an indexed family of W*-algebras, $M(A)=\{M_n(A)\}_n$.
A completely positive map 
$f:A\to B$ is taken to the corresponding family of positive maps $M(f)=\{M_n(f):M_n(A)\to M_n(B)\}_n$. 
This gives the main result of \cite{rennela-staton-mfps31}: the functor $M$ is full and faithful, i.e. completely positive maps are in natural bijection with families of positive maps.\\

\begin{theorem}[\cite{rennela-staton-mfps31}]
\label{thm:CPNUff}
The functor $M:\wCPNU\to[\Mat,\wPNU]$ is full and faithful.\\
\end{theorem}

Faithfulness is obvious, since for any CP-map $f:A\to B$ we have $M(f)_1=f$. Proving fullness is more involved and requires the following lemma.\\

\begin{lemma}\label{cp-m2-lemma}
Consider two positive maps $f_n: M_n(B) \to M_n(A)$ and $f_1: B \to A$ of C*-algebras. The following conditions are equivalent:
\begin{enumerate}
\item
$\forall y\in M_n(B), v\in \CC^n.\ \ v^*(f_n(y))v = f_1(v^*yv)$
\item
$f_n = M_n (f_1)$.
\end{enumerate}
\end{lemma}

\ \\
The proof of the lemma~\cite{rennela-staton-mfps31} makes use of stabilizer states in $\CC^n$. 

\newcommand{\completely}[1]{#1_{\mathrm{C}}}
\newcommand{\catC}[1]{\completely{\mathbf{#1}}}
\section{Main result: W*-algebras are colimits of CP-maps}
\label{sec:wstar-colim}
\newcommand{\NPLF}{\ensuremath{\mathrm{NPLF}}}

This section gives our main contribution: we show that infinite dimensional W*-algebras are canonical colimits of matrix algebras.\\

Our first result is based on the representation of W*-algebras by their cones of positive 
linear functionals. 
We say that an (abstract) cone is a module for the semiring of
positive reals. Thus it is a set $X$ that is equipped with 
both the structure of a commutative monoid $(X,+,0)$
and a function $(-\cdot-):\Rpnz\times X\to X$ that is a group
homomorphism in each argument and
such that $rs\cdot x=r\cdot(s\cdot x)$, $1\cdot x=x$.
Most examples of cones arise as subsets of a larger vector space that are not sub\emph{spaces} per se
but merely closed under addition and multiplication by positive scalars. 
For example, the set of positive reals itself forms a cone. 
The positive elements of a C*-algebra also form a cone.\\

Let $\Cone$ be the category of cones and structure preserving functions between them.

For any W*-algebras $A$ and $B$, the set of 
normal positive maps $A\to B$ forms a cone: it is closed under addition, zero, and 
multiplication by positive scalars. (Formally, we can say that $\wCPNU$ is enriched in the category $\Cone$, equipped with 
the usual symmetric monoidal structure: composition is a cone-homomorphism in each argument. This also plays a role in~\cite{pagani-selinger-valiron-popl14}.)
In particular we have a functor
$\wPNU(-,\CC):\opp{\wPNU}\to \Cone$.\\

\begin{proposition}
\label{prop:NPLF-ff}
The normal positive linear functional functor $\wPNU(-,\CC): \opp{\wPNU} \to \Cone$ is full and faithful.
\end{proposition}
\begin{proof}
Fullness essentially comes from the fact that the closedness and completeness of the positive cone imply that every positive linear map $A_* \rightarrow \C$ is bounded (see e.g.~\cite{namioka} or \cite[Th. V.5.5(ii)]{schaefer-top-vect}). We refer the interested reader to Appendix \ref{appendix:proof-NPLF-ff} for a detailed proof.
\end{proof}

We define a category $\NatCP$ as a full subcategory of $\wCPNU$ whose objects are 
W*-algebras of the form $M_n$. 
We consider the functor $\opp\wCPNU\to [\NatCP,\Set]$ 
which takes a W*-algebra $A$ to the functor $M_n\mapsto \wCPNU(A,M_n)$. 

\hide{
\begin{corollary}\label{cor:homcone}
The hom-cone functor $\wCPNU(-,=):\opp{\wCPNU} \to [\NatCP,\Cone]$ is full and faithful.
\end{corollary}
\begin{proof}
By combining Prop.~\ref{prop:NPLF-ff} with Theorem~\ref{thm:cp-nat-basecat}, 
we have that the composite 
\[
H:\opp{\wCPNU}\to [\NatCP,\wCPNU]\to [\NatCP,\Cone]
\]
is full and faithful. 
Note that $H(A)(n)=\wPNU(M_n(A),\CC))$. 
Now, the positive functionals $M_n(A)\to \CC$ are in bijective correspondence with 
completely positive maps $A\to M_n$ (e.g.~\cite{stormer}). In fact this correspondence extends to an
isomorphism of functors, $H(A)\cong \wCPNU(A,M_{-}):\NatCP\to \Cone$. 
\end{proof}
}

\begin{theorem}
\label{thm:main-theorem}
The hom functor $\wCPNU(-,=): \opp{\wCPNU} \to [\NatCP, \cat{Set}]$ is full and faithful.
\end{theorem}
\begin{proof}
By combining Prop.~\ref{prop:NPLF-ff} with Theorem~\ref{thm:CPNUff},
we have that the composite 
\[
\wCPNU(-\otimes =,\CC):\opp{\wCPNU}\to [\opp{\Mat},\opp{\wPNU}]\to [\opp\Mat,\Cone]
\]
is full and faithful. 
Our first step is to show that the hom-cone functor
\[\wCPNU(-,=):\opp{\wCPNU}\to[\NatCP,\Cone]\]
is full and faithful.
Indeed, by elementary category theory,
for any functor
\[H:\opp\wCPNU\times \mathbb D\to\Cone\] 
if there is an
identity-on-objects functor 
$F:\opp\Mat\to \mathbb D$ and 
and a family of isomorphisms
\begin{equation}\wCPNU(M_n(A),\CC)\cong H(A,F(n))
\qquad\text{\ natural in $n\in \opp\Mat$ and $A\in\wCPNU$}
\label{eqn:nat-iso}
\end{equation}
then the transpose of $H$ 
\[
\opp{\wCPNU}\to[\mathbb D,\Cone]
\]
is full and faithful.

In particular, let $H$ be the restricted hom-functor
$H:\opp\wCPNU\times \NatCP \to\Cone$.
For $F:\opp\Mat\to\NatCP$, 
we first note that for any matrix $V:m\to n$ in $\Mat$ 
we have a completely positive map $V^*(-)V$, with reference to Choi's theorem \cite{choi}. 
To turn this into a contravariant functor 
$F:\opp{\Mat}\to\NatCP$,
we note that $\Mat$ is self-dual
with the isomorphism $\opp\Mat\to\Mat$ taking a matrix $V$ to its transpose $V^\top$.
We let $F(V)=(V^\top)^*(-)V^\top$.
The natural isomorphism~\eqref{eqn:nat-iso}
is now the standard bijection between 
states $M_n(A)\to \CC$ and completely positive maps $A\to M_n$
(see e.g.~\cite{stormer}).
Thus we can conclude that the hom-cone functor 
$\opp{\wCPNU}\to[\NatCP,\Cone]$
is full and faithful.

It remains to show that the hom-\emph{set} functor 
\[\opp{\wCPNU}\to[\NatCP,\Set]\]
is full and faithful.
We must show that if a family of functions
$\phi_n:\wPNU(B,M_n)\to \wPNU(A,M_n)$ between 
hom-sets is natural in $n\in\NatCP$ then each $\phi_n$ is necessarily a cone homomorphism, 
i.e.~that
$\phi_n(\lambda.f)=\lambda.\phi_n(f)$
and 
$\phi_n(f+g)=\phi_n(f)+\phi_n(g)$. 
The first fact, $\phi_n(\lambda.f)=\lambda.\phi_n(f)$, comes immediately from naturality with respect to the CP-map
$M_n\to M_n$ given by scalar multiplication with the scalar $\lambda$. 
For the second fact, $\phi_n(f+g)=\phi_n(f)+\phi_n(g)$, we use a characterization of pairs of maps $A\to M_n$.
Let $j:M_{2n}\to M_{2n}$ be the idempotent completely positive map
$j(\begin{smallmatrix}a&b\\c&d\end{smallmatrix})=(\begin{smallmatrix}a&0\\0&d\end{smallmatrix})$.
We have a bijection
\begin{equation}
\wCPNU(A,M_n)\times \wCPNU(A,M_n)\ \cong\ \{h\in\wCPNU(A, M_{2n})~|~ h=j\cdot h\}\text.
\label{eqn:bij}
\end{equation}
This bijection takes a pair of maps 
$f,g:A\to M_n$ 
to the map $h:A\to M_{2n}$ with 
$h(a)=(\begin{smallmatrix}f(a)&0\\0&g(a)\end{smallmatrix})$. 
Under the bijection \eqref{eqn:bij}, we can understand addition in the cone $\wCPNU(A,M_n)$ as composition with the
CP map $t:M_{2n}\to M_n$ given by
$t(\begin{smallmatrix}a&b\\c&d\end{smallmatrix})=a+d$,
and so, since~$\phi$ is natural with respect to~$t$, the addition structure of the cone is preserved by each $\phi_n$.

Here is a higher level account of the previous paragraph. 
Let $\FDCP$ be the category of all finite dimensional C*-algebras and completely positive maps between them.
We have an equivalence of categories $[\NatCP, \cat{Set}]\simeq [\FDCP,\cat{Set}]$,
in other words, the Karoubi envelope of $\NatCP$ contains $\FDCP$ (e.g.~\cite{selinger-karoubi,heunen-kissinger-selinger}).
Now $\FDCP$ has a full subcategory $\CFDCP$, the \emph{commutative} finite dimensional C*-algebras and completely positive maps between them. 
In fact, this category $\CFDCP$ of commutative C*-algebras is equivalent to the Lawvere theory for abstract cones (c.f.~\cite[Prop.~4.3]{furber-jacobs}), 
so the category $\Cone$ of cones is a full subcategory of the functor category $[\CFDCP,\Set]$.
So natural maps in $[\NatCP,\cat{Set}]$ are, 
in particular, cone homomorphisms.\end{proof}

As discussed in the introduction, this theorem means that a W*-algebra can be understood as the canonical colimit of a diagram of matrix algebras and completely positive maps. At this point, it is important to stress that our result is about colimits in the opposite category $\opp{\wCPNU}$ of W*-algebras and completely positive maps, and therefore about limits in the category $\wCPNU$ of W*-algebras and completely positive maps.

\section{Some remarks on topological vector spaces}
\label{sec:cp-tvect}
It is natural to wonder whether we can abstract from the setting of the theory of W*-algebras and evolve to the larger scope of the theory of topological vector spaces.\\

Recall that a topological vector space $X$ over a topological field $K$ is a vector space whose addition $X^2 \to X$ and scalar multiplication $K \times X \to X$ are continuous with respect to the topology of $K$. All Hilbert spaces and Banach spaces are examples of such topological vector spaces. And, for every natural number $k$ and every topological vector space $X$, the $k \times k$ matrices whose entries are in $X$ also form a topological vector space, $M_k(X)$, where the topology is the product topology on $X^{k^2}$. 

A continuous K-linear map $f: X\to Y$ between topological vector spaces over the field $K$ therefore always extends to a continuous map $M_n(f):M_n(X)\to M_n(Y)$ for every natural number $n \in \N$, essentially by viewing $M_n(f)$ as $f \times \cdots \times f$, with $n^2$ factors. 

The category of topological vector spaces over a given topological field $K$ is commonly denoted by $\cat{TVS}_K$ or $\cat{TVect}_K$, taking topological vector spaces over $K$ as objects and continuous K-linear maps as arrows. We will additionally restrict to topological vector spaces over $\C$ whose
bornology (i.e. ideal of bounded sets \cite[\S I.5]{schaefer-top-vect}) is definable by a norm, a property which is true for C*-algebras and W*-algebras.\\

Consider a subcategory $\cat{V}$ of $\cat{TVect}_\C$ closed under matrix algebras, i.e. satisfying
\begin{equation}
\CC\in\cat V\quad\text{and }\quad A\in\cat V\implies M_n(A)\in\cat V\text.
\label{eqn:closed-under-matrices}
\end{equation}
We will call $\catC V$ the closure of the category $\cat{V}$ under matrices of morphisms. Then, one obtains the following theorem  analogous to Theorem \ref{thm:CPNUff}, in line with \cite{rennela-staton-mfps31}.\\

\begin{theorem}
Consider a subcategory $\cat{V}$ of $\cat{TVect}_\C$ 
satisfying~\eqref{eqn:closed-under-matrices} and such that the matrices functor
\[\catC V\times \Mat\to\cat{TVect}_\C
\]
factors through $\cat V$. 
It induces a full and faithful functor $\catC {V} \to [\Mat,\cat{V}]$.\\ 
\end{theorem}

From there we can build representations for some of the categories of topological vector spaces introduced in the literature \cite{alfsen-shultz,furber-thesis,schaefer-top-vect}. First, we will recall some definitions and then state the representation theorem associated to them.\\

\newcommand{\Banach}{\cat{Banach}}
\newcommand{\OpSp}{\cat{OpSpace}}
\newcommand{\OpSys}{\cat{OpSystem}}
\newcommand{\WkOpSys}{\cat{Wk}\OpSys}
\newcommand{\OUS}{\cat{OUS}}
\newcommand{\WkOUS}{\cat{Wk}\OUS}
\newcommand{\Predual}{\cat{Predual}}
\newcommand{\BNS}{\cat{BNS}}

A Banach space is a complete normed vector space. A (concrete) operator space is a closed subspace of a C*-algebra, or alternatively a Banach space given together with an isometric embedding into the space of all bounded operators on some Hilbert space $H$ \cite{pisier-intro-opsp}. We define $\Banach$ to be the category of Banach spaces and bounded maps, i.e. linear maps $\varphi$ such that $\exists k. \Vert\varphi (x)\Vert \leq k \cdot \Vert x \Vert$, and $\OpSp$ to be the category of operator spaces and completely bounded maps \cite{paulsen-CP-book}, i.e. linear maps $\varphi$ such that the norm $\Vert \varphi \Vert_\text{cb} := \sup_n \Vert M_n(\varphi)\Vert$ is finite, i.e. $\Vert M_n(\varphi)\Vert$ is bounded by a constant which does not depend on $n$. Alternatively, one could consider contractive maps, i.e. linear maps $\varphi$ such that $\Vert \varphi \Vert \leq 1$, and completely contractive maps, i.e. linear maps $\varphi$ such that $\Vert M_n(\varphi)\Vert \leq 1$ for every natural number $n$. Then, the category $\OpSp$ of operator spaces is a full subcategory of the category $\Banach_\text{CP}$ of Banach spaces and completely bounded maps.\\

With those definitions in mind, we can simply explain why we needed to make an additional requirement on the structure of the topological vector spaces that we are considering. In short, complete boundedness is not formalizable in terms of topology. In detail, an abstract operator space is a compatible choice of norms on $E$, $M_2(E)$, $M_3(E)$, $\cdots$ for a Banach space $E$. No matter which norm we choose, the topology is always the same, so it is always the case that $M_n(f)$ is continuous, and therefore bounded. Then, an operator space structure is a \textit{choice} of norm generating the bornology for each $n$, and completely bounded maps are those for which the norm of the maps $M_n(f)$ is bounded by some constant not depending on $n$. In the case where $E$ is equipped with a C$^*$-algebra structure, we take the norms to be the C$^*$-norms of each $M_n(E)$, and this defines what the norms are uniquely in the case of a concrete operator space. In the case of completely contractive maps the norm bound is automatically independent of $n$ (it is $1$) but one still needs a choice of norm on each $M_n(E)$ and therefore cannot stick to topology alone.\\

An order-unit space $(E,E_+,u)$ is an ordered vector space $(E,E_+)$ equipped with a strong Archimedian unit $u \in E_+$ \cite[Def. 1.12]{alfsen-shultz}. An operator system is an involutive vector space $V$ such that the vector space $M_n(V)$ of $n$-by-$n$ matrices whose entries are in $V$ is an order-unit space, or alternatively a closed subspace of a unital C*-algebra which contains $1$. We define $\OUS$ to be the category of order-unit spaces and unit-preserving positive maps between them, and $\OpSys$ to be the category of operator systems and unit-preserving completely positive maps between them. Then the category $\OpSys$ of operator systems is a full subcategory of the category $\OUS_\text{UCP}$ of order-unit spaces and unit-preserving completely positive maps.\\

\begin{theorem}
The following matrix functors $M$, taking a topological vector space $X$ to a functor $M(X): n \mapsto M_n(X)$, are full and faithful 
$$M:\OpSp \to [\Mat, \Banach] \qquad \qquad M:\OpSys \to [\Mat, \OUS]$$
\end{theorem} 
 
\section{Further work}
The present work only covers one specific aspect of the approximation of infinite-dimensional operator algebras by finite-dimensional ones. Determining how our construction relates to other approaches, like Tobias Fritz' perspective on infinite-dimensional state spaces \cite{fritz-infinite}, is a topic worth investigating.\\

Moreover, Day's construction \cite{day-convolution} provides a canonical way to extend the tensor product of a base category, like $\NatCP$, to 
a tensor product of presheaves (see also \cite{malherbe-scott-selinger}). Thus one can define a tensor product $A \star B$ of two W*-algebras $A$ and $B$ as the unique extension of the standard tensor product of matrix algebras that preserves colimits of CP-maps in each argument. This seems related to the way that Grothendieck defined tensor products on the category of Banach spaces \cite{grothendieck-tensors},
by starting with the category of finite\-dimensional normed spaces 
and a tensor defined on there.\\

Finally, although it has been established that every vector space is a filtered colimit of finite dimensional spaces, it is unclear whether there is an analogous characterization for W*-algebras.

\paragraph*{Acknowledgements}
The authors would like to thank Ross Duncan, Tobias Fritz, Shane Mansfield, Isar Stubbe, Bram Westerbaan and Fabio Zanasi for helpful discussions, and the anonymous referees for their comments and suggestions. The research leading to these results has received funding from the European Research Council under the European Union's Seventh Framework Programme (FP7/2007 -- 2013) / ERC grant agreement n. 320571, and a Royal Society University Research Fellowship. Robert Furber was supported by Project 4181 -- 00360 from the Danish Council for Independent Research.

\bibliographystyle{eptcs}
\bibliography{qpl16-eptcs}

\begin{appendix}
\section{Proof of Proposition \ref{prop:NPLF-ff}}
\label{appendix:proof-NPLF-ff}
\newcommand{\auxproof}[1]{
\ifignore\mbox{}\newline
\textbf{PROOF:} \dotfill\newline
{\it #1}\mbox{}\newline
\textbf{ENDPROOF}\dotfill
\fi}

\newcommand{\cohom}{\ensuremath{H^m(\overline{Y}, \mathbf{Q}_\ell)}}
\newcommand{\codim}{\ensuremath{\mathrm{codim}}}
\newcommand{\gr}{\ensuremath{\mathrm{gr}}}
\newcommand{\Rgeq}{\ensuremath{\mathbb{R}_{\geq 0}}}
\newcommand{\Hil}{\ensuremath{\mathcal{H}}}
\newcommand{\Prob}{\ensuremath{\mathbb{P}}}
\newcommand{\E}{\ensuremath{\mathbb{E}}}
\newcommand{\Var}{\ensuremath{\mathrm{Var}}}
\newcommand{\powerset}{\ensuremath{\mathcal{P}}}

\newcommand{\WeakDual}{\ensuremath{\mathrm{WeakDual}}}
\newcommand{\Herm}{\ensuremath{\mathrm{Herm}}}

\newcommand{\Obj}{\ensuremath{\mathrm{Obj}}}
\newcommand{\Arr}{\ensuremath{\mathrm{Arr}}}

\newcommand{\WStarP}{\wPNU}
\newcommand{\PoVect}{\cat{PoVect}}
\newcommand{\BNSP}{\cat{BBNS}_{P}}

\newcommand{\blank}{\ensuremath{\mbox{-}}}
\newcommand{\sembrack}[1]{[\![#1]\!]}

\newcommand{\braket}[2]{\ensuremath{\langle #1 | #2 \rangle}}
\newcommand{\ketbra}[2]{\ensuremath{ | #1 \rangle \langle #2 | }}

We'll consider the functor $\NPLF : \WStarP \rightarrow \Cone$ that assigns to each W$^*$-algebra its cone of normal (equivalently ultraweakly continuous) positive linear functionals, essentially defined such that $\NPLF = \WStarP(\blank,\C)$. We want to show it is full, using the fact that the closedness of the positive cone implies that every positive linear map $A_* \rightarrow \C$ is bounded \cite[Th. V.5.5(ii)]{schaefer-top-vect}.\\

To make the proof smoother, we start from the observation that if we have a partially ordered vector space $E$, it has a positive cone $E_+$, and positive (equivalently monotone) maps of partially ordered vector spaces $f: E \rightarrow F$ restrict to cone maps $E_+ \rightarrow F_+$, defining a functor $\blank_+ : \PoVect \rightarrow \Cone$, where $\PoVect$ is partially ordered vector spaces that are generated by their positive cone (also known as directed partially ordered vector spaces because directedness in the usual sense is equivalent to this property).

\begin{proposition}
The functor $\blank_+ : \PoVect \rightarrow \Cone$ is full and faithful.
\end{proposition}

\begin{proof}
\textbf{Faithfulness}\qquad
Let $f,g : E \rightarrow F$ and $f_+ = g_+$, \emph{i.e.} for all $x \in E_+$, we have $f(x) = g(x)$. Then since $E$ is that span of $E_+$, we have that each element $x$ of $E$ is expressible as $x_+ - x_-$ for $x_+,x_- \in E_+$. Then
\[
f(x) = f(x_+ - x_-) = f(x_+) -f(x_-) = g(x_+) - g(x_-) = g(x).
\]

\paragraph{Fullness} \qquad
Suppose $g : E_+ \rightarrow F_+$ is a cone map, \emph{i.e.} a monoid homomorphism preserving multiplication by a nonnegative real number. We extend it to a linear map $f : E \rightarrow F$ as follows. Let $E \ni x = x_+ - x_-$ as in the previous part. Define $f(x) = g(x_+) - g(x_-)$. We first show that this is well defined, so let $y_+,y_-$ be elements of $E_+$ such that $x = y_+ - y_-$. Then
\begin{align*}
&y_+ - y_- = x_+ - x_- \implies y_+ + x_- = x_+ + y_-\\
\implies &g(y_+ + x_-) = g(x_+ + y_-) \implies g(y_+) - g(y_-) = g(x_+) - g(x_-),
\end{align*}
which shows that $f(x)$ is independent of the decomposition into positive parts that has been chosen. 

Since each positive element $x$ can be expressed as $x - 0$, we have that $f(x) = g(x)$ on positive elements, and in particular that $f$ preserves the positive cone and $f_+ = g$. We therefore only need to show that $f$ is in fact linear.

So now let $x = x_+ - x_-$ and $y = y_+ - y_-$. Then
\begin{align*}
f(x + y) &= f((x_+ + y_+) - (x_- + y_-)) \\
 &= g(x_+ + y_+) - g(x_- + y_-) \\
 &= g(x_+) + g(y_+) - g(x_-) - g(y_-) \\
 &= (g(x_+) - g(x_-)) + (g(y_+) - g(y_-)) \\
 &= f(x) + f(y).
\end{align*}
It remains to show that $f$ preserves multiplication by a scalar $\alpha \in \R$. There are three cases, $\alpha = 0$, $\alpha > 0$ and $\alpha < 0$. The case that $\alpha = 0$ is trivial because $0 \in E_+$ and $g$ is a cone map, so preserves $0$. In the case that $\alpha > 0$, $\alpha(x_+ + x_-) = \alpha x_+ - \alpha x_-$ is still a decomposition into positive elements, so
\begin{align*}
f(\alpha x) &= g(\alpha x_+) - g(\alpha x_-) \\
 &= \alpha (g(x_+) - g(x_-)) \\
 &= \alpha f(x).
\end{align*}
In the case that $\alpha < 0$, then $(-\alpha x_-) - (-\alpha) x_+ = x$ is a decomposition into positive elements, so
\begin{align*}
f(\alpha x) &= g(-\alpha x_-) - g(-\alpha x_+) \\
 &= -\alpha g(x_-) - (-\alpha)g(x_+) \\
 &= \alpha (g(x_+) - g(x_-)) \\
 &= \alpha f(x).
\end{align*}
\end{proof}

The predual $A_*$ of a W$^*$-algebra $A$ can be identified with the ultraweakly continuous linear functionals \cite[I.3.3 Theorem 1 (iii)]{dixmiervna}, and the positive elements with the normal positive linear functionals \cite[I.4.2 Theorem 1]{dixmiervna}. In particular, the map $\zeta_A : A \rightarrow (A_*)^*$ defined by $\zeta_A(a)(\phi) = \phi(a)$ is an isomorphism.

Hermitian linear functionals are those functionals $\phi$ such that $\overline{\phi(a^*)} = \phi(a)$ for all $a \in A$ \cite[1.1.10]{dixmier}, and they are the $\R$-span of the positive ones \cite[Theorem 12.3.3]{dixmier}. Every complex normal linear functional can be decomposed into real and imaginary parts, which are Hermitian, so the $\C$-span of the positive normal functionals is the ultraweakly continuous functionals. The Hermitian elements of the predual form a base-norm space \cite[Corollary 2.96]{alfsen-shultz}. So for each W$^*$-algebra we define $\Herm(A)$ to be this base-norm space. If we take $\BNSP$ to be the category of Banach base-norm spaces with positive maps, and we define a functor $\Herm : \WStarP \rightarrow \BNSP$, for $f : A \rightarrow B$ a positive map of W$^*$-algebras as $\Herm(f)(\phi) = \phi \circ f$,
for $\phi \in \Herm(B)$. 

We have that $\NPLF = \blank_+ \circ \Herm$, so we reduce to showing that $\Herm : \WStarP \rightarrow \PoVect$ is full and faithful.

\begin{theorem}
The functor $\Herm$ is full and faithful.
\end{theorem}

\begin{proof}
We first prove it is faithful as follows. Let $f, g : A \rightarrow B$ be positive ultraweakly continuous (or normal) maps between W$^*$-algebras, such that $\Herm(f) = \Herm(g)$. If $f \neq g$, there is an $a \in A$ such that $f(a) \neq f(b)$. Since $B$ is separated by normal states, there is a $\phi \in \Herm(B)$ such that $\phi(f(a)) \neq \phi(f(b))$, and therefore $\Herm(f)(\phi) \neq \Herm(g)(\phi)$, contradicting the assumption that $\Herm(f) = \Herm(g)$. Therefore $\Herm$ is faithful.\\

We now prove the fullness. Let $g : \Herm(B) \rightarrow \Herm(A)$ be a positive map. Let $a : \Herm(A) \rightarrow \R$ be a positive linear map. By Theorem \cite[Th. V.5.5(ii)]{schaefer-top-vect} it is bounded, and so $\zeta_A^{-1}(a)$ exists, an element of $A$, which is necessarily positive. Since $a \circ g : \Herm(B) \rightarrow \R$ is also positive, and therefore defines a positive element of $B$ under $\zeta_B$, we have a function mapping positive elements of $A$ to positive elements of $B$, defined as $\zeta_B^{-1} \circ (\blank \circ g) \circ \zeta_A$. We show that it is a cone map, and therefore extends to a positive linear map $A \rightarrow B$ as follows.\\

\noindent \textbf{Preservation of zero} \qquad Since each linear functional maps $0$ to $0$, $\zeta_A(0)$ is the constant zero map. Precomposing with $g$ produces another constant zero map $\Herm(B) \rightarrow \R$, so $\blank \circ g$ maps zero to zero.

\noindent \textbf{Additivity} \qquad Addition in $A$ and $B$ corresponds to pointwise addition. Let $a,b : \Herm(A) \rightarrow \R$ be positive linear maps. Then for each $\phi \in \Herm(B)$, we have
\begin{align*}
((a + b) \circ g)(\phi) &= (a+b)(g(\phi)) \\
 &= a(g(\phi)) + b(g(\phi)) \\
 &= (a \circ g)(\phi) + (b \circ g)(\phi) \\
 &= (a \circ g + b \circ g)(\phi).
\end{align*}
\textbf{Preservation of positive multiplications} \qquad
Let $a : \Herm(A) \rightarrow \R$ be a positive linear map, $\alpha \in \Rgeq$ and $\phi \in \Herm(B)$. Then 
\begin{align*}
((\alpha a) \circ g)(\phi) &= (\alpha a)(g(\phi)) \\
 &= \alpha a(g(\phi)) \\
 &= \alpha (a \circ g)(\phi) \\
 &= (\alpha (a \circ g))(\phi).
\end{align*}
We therefore have a positive linear map $f : A \rightarrow B$, but have not yet shown that it is normal or that $\Herm(f) = g$. 

We first show that $\blank \circ f = g$ on all $\phi \in B_*$. This will imply that $f$ is ultraweakly continuous and so that $\Herm(f) = g$. So let $\phi \in B_*$ and $a \in A$. Then
\begin{align*}
(\blank \circ f)(\phi)(a) &= \phi(f(a)) \\
 &= \zeta_B(f(a))(\phi) & \text{(definition of $\zeta$)}\\
 &= \zeta_B(\zeta_B^{-1}((\blank \circ g)(\zeta_A(a))))(\phi) & \text{(definition of $f$)} \\
 &= (\zeta_A(a) \circ g)(\phi) \\
 &= \zeta_A(a)(g(\phi)) \\
 &= g(\phi)(a).
\end{align*}

By \cite[IV.2.2]{schaefer-top-vect} this implies that $f$ is weak-* continuous, and the definition of $\Herm$ implies $\Herm(f) = g$.
\end{proof}

Considering that the composite of two full and faithful functors is full and faithful, one obtains the following corollary.

\begin{corollary}
The functor $\NPLF$ is full and faithful.
\end{corollary}

\end{appendix}
\end{document}